\documentclass[12pt]{article}
\usepackage[margin=2in,includefoot,footskip=12pt,]{geometry}
\usepackage{layout}
\usepackage[utf8]{inputenc}
\usepackage{amsmath}
\usepackage{amsfonts}
\usepackage{amssymb}
\usepackage{amsthm,fullpage}
\usepackage{graphics}
\usepackage{float,graphicx}
\usepackage{tikz}
\usepackage[toc,page]{appendix}
\usepackage{hyperref}
\usepackage{xcolor}
\usepackage[export]{adjustbox} 
\usepackage{subfig}

\usepackage{makecell}
	
	\newtheorem{theorem}{Theorem}[section]
	\newtheorem{corollary}{Corollary}[section]
	\newtheorem{lemma}{Lemma}[section]
	
	\newtheorem{remark}{Remark}[section]
	\newtheorem{example}{Example}[section]

	\newtheorem{conj}{Conjecture}[section]
	
	\DeclareMathOperator{\spec}{spec}
       \renewcommand{\Re}{\operatorname{Re}}

	\title{ Bounds and extremal graphs for the energy of complex unit gain graphs\footnote{Some of the results of this article are part of the PhD thesis of first author}}

\author{Aniruddha Samanta \thanks{Theoretical Statistics and Mathematics Unit, Indian Statistical Institute, Kolkata-700108, India. Email: aniruddha.sam@gmail.com}\  \and M. Rajesh Kannan\thanks{Department of Mathematics, Indian Institute of Technology Hyderabad, Hyderabad 502285, India. Email: rajeshkannan@math.iith.ac.in, rajeshkannan1.m@gmail.com }}

\date{\today}
\begin{document}
\maketitle
\baselineskip=0.25in

\begin{abstract}

		A complex unit gain graph ($ \mathbb{T} $-gain graph), $ \Phi=(G, \varphi) $ is a graph where the gain function $ \varphi $ assigns a unit complex number to each orientation of an edge of $ G $ and its inverse is assigned to the opposite orientation. The associated adjacency matrix $ A(\Phi) $ is defined canonically. The energy $ \mathcal{E}(\Phi) $ of a $ \mathbb{T} $-gain graph  $ \Phi $ is the sum of the absolute values of all eigenvalues of $ A(\Phi) $. For any connected triangle-free $ \mathbb{T} $-gain graph $ \Phi $ with the minimum vertex degree $ \delta$, we establish a lower bound $ \mathcal{E}(\Phi)\geq 2\delta$ and characterize the equality. Then, we present a relationship between the characteristic and the matching polynomial of $ \Phi $. Using this, we obtain an upper bound for the energy $ \mathcal{E}(\Phi)\leq 2\mu\sqrt{2\Delta_e+1} $ and characterize the classes of graphs for which the bound sharp, where $ \mu$ and $ \Delta_e$ are the matching number and the maximum edge degree of $ \Phi $, respectively. Further, for any unicyclic graph $ G $,  we study the gains for which the gain energy  $ \mathcal{E}(\Phi) $ attains the maximum/minimum among all $ \mathbb{T} $-gain graphs defined on $G$.		
		
	\end{abstract}



\section{Introduction}
Let $ G $ be a simple graph with vertex set $ V(G)=\{v_1, v_2, \dots, v_n \} $ and edge set $ E(G) $. We denote $ v_s \sim v_t $ if the vertices $ v_s $ and $ v_t $ are adjacent,  and the edge between them is denoted by $ e_{s,t} $. The \emph{adjacency matrix } $ A(G) $ of a graph $ G $ is a symmetric $ n\times n $ matrix whose $ (s,t)th $ entry is $ 1 $ if $ v_s\sim v_t $ and zero otherwise. Let the multiset $ \{\lambda_1, \dots, \lambda_n\} $ be the eigenvalues (or the spectrum) of $ A(G) $. The \emph{energy} of the graph $ G $ is defined as 

\begin{equation*}
	\mathcal{E}(G)=\sum\limits_{k=1}^{n}|\lambda_k|.
\end{equation*}
Graph energy was systematically studied by Gutman \cite{Gutman} in 1978, which was motivated by the notion of total $ \pi $-electron energy from chemistry. Graph energy has several applications in science and engineering. For a detailed study on graph energy and its applications, we refer to \cite{book_Gutman}. This paper aims to study the energy of $\mathbb{T}$-gain graphs. 
A \emph{directed graph(or digraph)} $ X $ is an order pair $ (V(X), E(X)) $, where\break $ V(X)=\{ v_{1}, v_{2}, \dots,v_{n}\} $ is the vertex set and $ E(X) $ is the directed edge set. A directed edge from the vertex $ v_{p} $ to the vertex $ v_{q} $ is denoted by $ \vec{e}_{p,q} $.

From a simple graph $G$, by orienting each undirected edge $ e_{p,q} \in E(G)$ in two opposite directions, namely $ \vec{e}_{p,q}$ and $ \vec{e}_{q,p}$, we get a digraph.  Let $ \vec{E}(G)=\{\vec{ e}_{p,q}, \vec{e}_{q,p}: e_{p,q}\in E(G) \} $ and $ \mathbb{T}=\{ z \in \mathbb{C}: |z|=1\} $. A \emph{complex unit gain graph} (simply, $ \mathbb{T} $-gain graph) on a simple graph $ G $ is a pair $ (G, \varphi) $, where $ \varphi: \vec{E}(G) \rightarrow \mathbb{T} $ is a mapping  such that $ \varphi( \vec{e}_{p,q}) =\varphi(\vec{e}_{q,p})^{-1}$. A $ \mathbb{T} $-gain graph $ (G, \varphi) $ is  denoted by $ \Phi $.  For more details about the $\mathbb{T}$-gain graphs, we refer to \cite{Tgain9,Tgain8(line),Reff2016,our-paper3, Tgain3(rank),Zaslav}.

The \emph{ adjacency matrix} of $ \Phi$ is the Hermitian  matrix $ A(\Phi)=(a_{p,q})_{n \times n }$  defined as follows:
$$a_{p,q}=\begin{cases}
	\varphi(\vec{e}_{p,q})&\text{if } \mbox{$v_p\sim v_q$},\\
	0&\text{otherwise.}\end{cases}$$

Let the multiset $ \{\lambda_1, \dots, \lambda_n\} $ be the eigenvalues of $ A(\Phi) $ (or the spectrum of $ \Phi $), and is denoted by $\spec(\Phi)$. The energy of $ \Phi $, denoted by  $ \mathcal{E}(\Phi) $, is defined by $\sum\limits_{j=1}^{n}|\lambda_j| $. The spectral radius of $ \Phi $, denoted by $ \rho(\Phi) $, is the maximum of absolute values of all the eigenvalues of $\Phi$. The adjacency matrices of  $\mathbb{T}$-gain graphs include several known classes of adjacency: Signed adjacency matrices \cite{signed-adj}, Hermitian adjacency matrices are mixed graphs of first kind and second kind \cite{ guo-mohar-jgt, Lie, mohar-secondkind}, complex adjacency matrices with entries are from the set $\{\pm 1, \pm i\}$ \cite{bap-kal-pat-weighted} and so on.

The \emph{degree} of a vertex $ v$  of a graph $G$  is the number of edges incident with $ v $, denoted by $ d(v) $. Then, the maximum and the minimum vertex degree of $ G $ are denoted by  $ \Delta(G) $ and $ \delta(G) $, respectively. In \cite{Xiaobin-Ma}, Ma proved the following bound for the energy of a simple graph in terms of the minimum vertex degree.

\begin{theorem}[{\cite[Theorem 1.2]{Xiaobin-Ma}}]\label{Th1.1}
	Let $ G $ be a connected simple graph with the minimum vertex degree $\delta$. Then $\mathcal{E}(G)\geq 2\delta$ and equality occurs if and only if $ G $ is a multipartite graph with equal partition size.
\end{theorem}

Later, Oboudi \cite{Mohammad_Reza_Oboudi} gave a simple proof of the above result. Our first objective in this article is to establish a lower bound for the energy of a connected triangle-free $ \mathbb{T} $-gain graph in terms of the minimum vertex degree and to characterize the sharpness of the inequality with the help of vertex energy. First, we show that the counterpart of Theorem \ref{Th1.1} need not be true for $ \mathbb{T} $-gain graphs (See, example \ref{ex2}). Surprisingly, the counterpart of Theorem \ref{Th1.1} holds for triangle-free $ \mathbb{T} $-gain graphs. Moreover, we characterize the $\mathbb{T}$-gain graphs for the bound is sharp. This is done in Section  \ref{lower_bound}.  

The \emph{degree of an edge $ e$} of $ G $ is the number of edges incident with it. Let $ \Delta_e(G) $ denote the maximum edge degree of $ G $. A \emph{matching} in a simple graph $ G $ is a set of independent edges of $ G $. The \emph{matching number} $ \mu(G) $ is the cardinality of a maximum matching of $ G $. Wang and Ma \cite{Wang-Ma} established a lower bound $ \mathcal{E}(G)\geq 2\mu(G)$ for bipartite graph $ G $. In \cite{Wong-Wang-Chu}, Dein Wong et al. proved the result for any simple graph $ G $. Later, this result has been extended for mixed graphs \cite{Wei-Li} and $ \mathbb{T} $-gain graphs \cite{our-paper3}. Yingui Pan et al. \cite{Pan-Chen-Li} obtained the following upper bound for energy $ \mathcal{E}(G) $ in terms of $ \mu(G) $.  Let $P_n$ denote the path graph on $n$ vertices.

\begin{theorem}[{\cite[Theorem 1]{Pan-Chen-Li}}]\label{Th1.2}  
	Let $ G $ be a simple graph with the matching number $ \mu$ and maximum edge degree $\Delta_e $. Then 
	\begin{itemize}
		\item[(i)] If $ \Delta_e$ is even, then $ \mathcal{E}(G)\leq 2\mu\sqrt{2\Delta_e+1} $. Equality occurs if and only if $G$ is the disjoint union of $ \mu$ copies of $ P_2$ and some isolated vertices.
		\item [(ii)]  If $ \Delta_e$ is odd, then $ \mathcal{E}(G)\leq \mu\left( \sqrt{b+2\sqrt{b}}+\sqrt{b-2\sqrt{b}}\right) $, where $ b=2(\Delta_e+1) $. Equality occurs if and only if $G$ is $ \mu$ copies of $ P_3$ and some isolated vertices.
	\end{itemize} 
\end{theorem}

In \cite{S_Akbari}, Akbari et al. obtained an alternative simple proof of Theorem \ref{Th1.2}. Note that this result was first proved by Tian and Wong \cite{F.Tian-D.Wong} for triangle-free graphs. Our second objective of this article is to extend the above bound for the $ \mathbb{T} $-gain graphs. The main idea of the proof of Akbari et al. depends on graph decomposition. This decomposition technique is not helpful for the $\mathbb{T}$-gain graphs (see Example \ref{deco-coun-exam}). We develop an alternate graph decomposition technique for the $\mathbb{T}$-gain graphs. Also, we establish an integral representation for the $\mathbb{T}$-gain graphs and a relation between the characteristic and matching polynomials of $\mathbb{T}$-gain graphs. Using these ideas, we extend Theorem \ref{Th1.2} for the $\mathbb{T}$-gain graphs. This is done in Section \ref{upper_bound}.
	

	 Let $ \mathcal{T}_G $ denote the collection of all $\mathbb{T}$-gain graphs defined on $G$. A natural question that arises here is the following: For which simple graph $ G $, the energy $ \mathcal{E}(G) $ is minimum/maximum among its $ \mathbb{T} $-gain graphs $ \mathcal{T}_G $ ?
	 
	 For a graph $G$, the Seidel adjacency matrix is defined as follows: $S(G) = J - 2A(G) - I$, where $J$ is the all one's matrix and  $A(G)$ is the adjacency matrix of $G$. The Seidel adjacency matrix of any graph $G$ can be viewed as a signed adjacency matrix of a complete signed graph. One of the conjectures regarding the Seidel energy given by Haemers \cite{Haemers}  can be rephrased as follows:  $ \mathcal{E}(K_{n})\leq \mathcal{E}(K_{n}, \psi) $, for any signed graph $ (K_{n}, \psi) $.  Akbari et al. affirmatively settled this conjecture \cite{Akbari}. In Section \ref{extremum_energy}, first, we show that the Haemers conjecture need not be true for a $ \mathbb{T} $-gain graph $(K_n, \psi) $. Moreover, in the case of $\mathbb{T}$-gain graphs, the extremal energy need not be attained for the graph $G$ (unlike the signed case). So it is natural to ask the following question: Identify the classes of graphs for which the extremal energy is the energy of $G$. In \cite{our-paper3}, we established that $ \mathcal{E}(K_{n,n})\leq \mathcal{E}(K_{n,n}, \varphi) $ for any $ \mathbb{T} $-gain graph $ (K_{n,n}, \varphi) $.
	 In Theorem \ref{Th5.1}, we show that for any unicyclic graph $ G $,  $ \mathcal{E}(G) $ attains extremum(either maximum or minimum) energy among all the $ \mathbb{T} $-gain graphs $ \Phi=(G, \varphi) $. The extremal families are also characterized. Two $ \mathbb{T} $-gain graphs $ \Phi_1 $ and $ \Phi_2 $ are \textit{ equienergetic} if $ \mathcal{E}(\Phi_1)=\mathcal{E}(\Phi_2)$.   If $G$ is a tree, then any two elements of $\mathcal{T}_G$ are equienergetic. In Theorem \ref{thm-equi-ener}, we show that if $G$ is a graph such that all its cycles are vertex disjoint, then $\mathcal{T}_G$ is not equienergetic. Finally, we conjecture that $\mathcal{T}_G$ is equienergetic if and only if $G$ is a tree. 
	 
	
\section{Preliminaries}\label{prelim}
	This section recalls some known definitions and results necessary for this article. Let $ G $ be an undirected simple graph. If $ G $ contains no triangle as an induced subgraph, it is called a \emph{triangle-free} graph.  Let us recall the following couple of results for undirected graphs.

\begin{theorem}[{\cite[Theorem 4.7]{book-Jukna}}](Mantel's theorem)\label{Th2.1}
	Any $ n $-vertices triangle-free graph has at most $ \lfloor \frac{n^{2}}{4} \rfloor $ edges.
\end{theorem}	
\begin{theorem} [{\cite[Theorem 2.38]{book-stanic}}]\label{Th2.2}
	Let $ G $ be a graph of $ m $ edges with the largest eigenvalue $ \lambda_1(G) $. Then $ \lambda_1(G)\leq \sqrt{m}.$
\end{theorem}
A \emph{clique} of a graph $ G $ is a complete subgraph. The \emph{clique number} of a graph $ G $ is the number of vertices of a clique with maximum size, and it is denoted by $ \omega(G) $. The following result is due to Nikiforov \cite{NIKIFOROV2009819}.

\begin{theorem}[{\cite[Theorem 2]{NIKIFOROV2009819}}]\label{Th2.3}
	Let $ G $ be a graph with no isolated vertices. Let $ \lambda_1(G), m,$ and $ \omega(G) $ be the largest eigenvalue, number of edges, and clique number of $ G $, respectively.
	Then
	\begin{equation*}
		\lambda_1(G)^{2}\leq \frac{2(\omega(G)-1)}{\omega(G)}m.
	\end{equation*}
	Equality holds if and only if one of the following occurs:
	\begin{itemize}
		\item[(i)]  $ \omega(G)=2 $ and $ G $ is complete bipartite graph.
		\item[(ii)] $ \omega(G)\geq 3 $ and $ G $ is a complete regular $ \omega(G)$-partite graph.
	\end{itemize}
\end{theorem}

Let $ G $ be a simple connected graph and $ v_i, v_j $ be any two vertices of $ G $. Then the \emph{distance} between $ v_i $ and $ v_j $ is the length of the shortest path in between $ v_i $ and $ v_j $. The \emph{diameter} of a graph $ G $ is the maximum distance between any two vertices in $ G $. The \emph{diameter} of a $ \mathbb{T} $-gain graph is the diameter of its underlying graph.
\begin{figure} [!htb]
	\begin{center}
		\includegraphics[scale= 0.60]{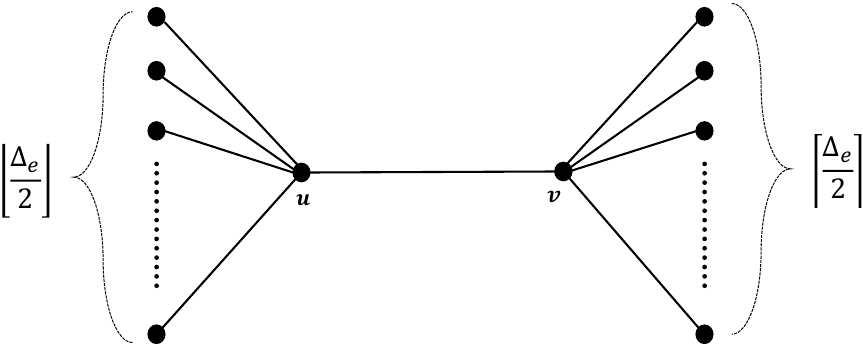}
		\caption{ Graph $ T_1 $} \label{fig3}
	\end{center}
\end{figure}

\begin{lemma}[{\cite[Lemma 2.1]{F.Tian-D.Wong}}]\label{Lm2.1}
	Let $ \mathcal{A}_{\Delta_e,3} $ be the set of all diameter $ 3 $ trees with the maximum edge degree $ \Delta_e $. Then for any $ T\in\mathcal{A}_{\Delta_e,3}  $, $ \mathcal{E}(T)\leq \mathcal{E}(T_1) $, where $ T_1\in \mathcal{A}_{\Delta_e,3}$ is shown in Figure \ref{fig3}. Equality occurs if and only if $ T=T_1 $.
\end{lemma}

\begin{lemma}[{\cite[Lemma 2.3]{F.Tian-D.Wong}}]\label{Lm2.2}
	Let $ T_1 $ be a tree given in Figure \ref{fig3}. Then $ \mathcal{E}(T_1)=2\sqrt{2\Delta_e+1} $ if $ \Delta_e $ is even and $ \mathcal{E}(T_1)=\sqrt{b+2\sqrt{b}}+\sqrt{b-2\sqrt{b}} $ if $ \Delta_e $ is odd, where $ b=2(\Delta_e+1)$.
\end{lemma}

 The maximum vertex degree, minimum edge degree, maximum edge degree and matching number of a $ \mathbb{T} $-gain graph $ \Phi=(G, \varphi) $ are defined as that of the underlying graph $ G $. Two $ \mathbb{T} $-gain graphs $ \Phi_1=(G, \varphi_1) $ and $ \Phi_2=(G, \varphi_2) $ are  \emph{switching equivalent} if there exist a unitary diagonal matrix $ U $ such that $ A(\Phi_2)=UA(\Phi_1)U^{*} $. We write $ \Phi_1 \sim \Phi_2$, if $ \Phi_1 $ and $ \Phi_2$ are switching equivalent. If $ \Phi_1 \sim \Phi_2 $ then $ \spec(\Phi_1)=\spec(\Phi_2) $. Let $ C $ be a cycle of $ n $ vertices, and the vertices are arranged as follows: $ C: v_1\sim v_2 \sim \cdots \sim v_n \sim v_1 $. Now the gain of $ \vec{C} $ is defined as $ \varphi(\vec{C})=\varphi(\vec{e}_{1,2})\varphi(\vec{e}_{2,3})\dots \varphi(\vec{e}_{n, n-1}) \varphi(\vec{e}_{n,1}) $.  For a complex number $\lambda$, let $\Re (\lambda)$ denote the real part of $\lambda$. For any cycle $ C $, instead of  $ \Re (\varphi(\vec{C}))$, we simply write $ \Re(\varphi(C)) $. A $ \mathbb{T} $-gain graph $ \Phi=(G, \varphi) $ is called \emph{balanced} if $ \varphi(C) =1$, for all cycle $ C $ in $ G $.  If $ \Phi $ is balanced, then $ \Phi \sim (G,1) $.  

Now, we recall a couple of results for the spectrum of $ \mathbb{T} $-gain graphs.

\begin{lemma}[{\cite[Corollary  3.1]{Our-paper-1}}]\label{Lm2.3}
	Let $ \Phi=(G, \varphi)$ be a $ \mathbb{T} $-gain graph with characteristic polynomial $ P_{\Phi}(x)= x^n+b_1(\Phi)x^{n-1}+b_2(\Phi)x^{n-2}+\cdots+b_n(\Phi)$. Then
	\begin{eqnarray} \label{eq5}
		b_{i}(\Phi)=\sum\limits_{H \in \mathcal{H}_{i}(G)}(-1)^{n(H)}2^{c(H)}\prod\limits_{C\in \mathcal{C}(H)}\Re(\varphi(C)),~~ \text{for}~ i=1,2, \dots, n
	\end{eqnarray}
	where $ \mathcal{H}_{i}(G) $ is the set of all elementary subgraphs with $ i $ vertices in $ G $ and $ n(H),~c(H),~\mathcal{C}(H) $ are the number of components, number of cycles and collection of cycles in $ H $, respectively.
\end{lemma}

\begin{theorem} [{\cite[Lemma 4.1, Theorem 4.4]{Our-paper-1}}] \label{Th2.4}
	Let $ \Phi=(G, \varphi) $ be any $ \mathbb{T} $-gain graph on a connected graph $ G $. Then $ \rho(\Phi) \leq \rho(G)$. Equality occurs if and only if either $ \Phi $ or $ -\Phi $ is balanced.
\end{theorem}

\begin{theorem}[{\cite[Theorem 4.6]{Our-paper-1}}]\label{Th2.5}
	Let $ \Phi $ be a $ \mathbb{T} $-gain graph on an underlying graph $ G $. Then $ \spec(\Phi)=\spec(G) $ if and only if $ \Phi $ is balanced.
\end{theorem}

\begin{theorem}[{\cite[Theorem 4.1]{Our-paper-1}}]\label{Th2.11}
	Let $\Phi$ be a bipartite $ \mathbb{T} $-gain graph. Then, the spectrum of $ \Phi $ is symmetric about the origin.
\end{theorem}

\begin{theorem}[{\cite[Theorem 4.5]{Our-paper-1}}] \label{Th26}
	If $ \Phi=(G, \varphi) $ be a connected $ \mathbb{T} $-gain graph, then
	\begin{itemize}
		\item[(1)] If $ G $ is bipartite, then $ \Phi $ is balanced implies $ -\Phi $ is balanced.
		\item[(2)] If $ \Phi $ is balanced implies $ -\Phi $ is balanced for some gain, then $ G $ is bipartite.
	\end{itemize}
\end{theorem}

\begin{theorem}[{\cite[Theorem 4.2]{our-paper3}}]\label{Th2.7}
	Let $ \Phi=(G, \varphi) $ be any $ \mathbb{T} $-gain graph on a connected bipartite graph $ G $. Then $ \Phi $ has exactly one positive eigenvalue if and only if $ \Phi $ is a balanced complete bipartite graph.
\end{theorem}

Arizmendi et al. \cite{Vertexenergy1} introduced the notion of the energy of a vertex for an undirected graph. Later,  in \cite{our-paper3},  we extended this notion for $ \mathbb{T} $-gain graphs. Let $ v_i \in V(\Phi) $. Then the \emph{energy of the vertex} $ v_i $ in $ \Phi $ is defined by $\mathcal{E}_{\Phi}(v_i):=|A(\Phi)|_{ii} $, where $ |A(\Phi)|_{ii}$ is the $ (i,i)th $ entry of the matrix $(A(\Phi)A(\Phi)^{*})^{\frac{1}{2}} $. Note that, $ \mathcal{E}(\Phi)=\sum\limits_{i=1}^{n}\mathcal{E}_{\Phi}(v_i) $.

\begin{lemma} [{\cite[Lemma 3.1]{our-paper3}}]\label{Lm2.4}
	Let $ \Phi $ be a $ \mathbb{T} $-gain graph with vertex set $\{ v_1, v_2, \cdots,  v_n\}  $. Then
	\begin{equation*}
		\mathcal{E}_{\Phi}(v_i)=\sum\limits_{j=1}^{n}P_{ij}|\lambda_j|, \text{  for } i=1,2, \dots n,
	\end{equation*}
	where $ P_{ij}=|q_{ij}|^{2} $ and $ Q=(q_{ij}) $ is the unitary matrix whose columns are the eigenvectors of $ A(\Phi) $ and $ \lambda_j $ is the $ j$-th  eigenvalue of $A(\Phi) $.	  	
\end{lemma}  	
The following integral representation will be useful. 
\begin{theorem}[{\label{Th2.9}\cite[Theorem 1]{Miodrag-Bozin-Gutman}}]
	Let $ X(t) $ be a polynomial of degree $ n $ with leading coefficient $ 1 $ and $ \lambda_1, \lambda_2, \dots, \lambda_n $ be the roots of $ X(t) $. Then 
	\begin{eqnarray}\label{eq8}
		\sum\limits_{j}|\Re(\lambda_j)|=\frac{1}{\pi}\int_{-\infty}^{\infty}\frac{1}{t^{2}}\log \left|t^{n}X\left(\frac{i}{t}\right)\right|dt.
	\end{eqnarray}
\end{theorem}

Next, we recall a couple of results from matrix theory, which will be used in the proofs of the main results.

\begin{theorem}[{\cite[Corollary 2.4]{Day-So}}]\label{Th2.10}
	Let $X= \begin{pmatrix}
		Y & A\\
		B& Z
	\end{pmatrix}$ be a partition matrix such that $ Y $ and $ Z $ are the square matrices. Then $ \mathcal{E}(X)\geq \mathcal{E}(Y) $. Equality holds if and only if $ A, B $, and $ Z $ are zero matrices.
\end{theorem}	

\begin{theorem}[{\cite{Cheng-Horn-Li}}]\label{Th2.8}
	Let $ C,~C_1$ and $C_2$ be three square complex matrices of order $ n $ such that \break $ C=C_1+C_2 $. If $ S_j(\cdot{}) $ is the $j$-th singular value of corresponding matrix, then \break $ \sum\limits_{p}S_{p}(C)\leq \sum\limits_{p}S_{p}(C_1)+\sum\limits_{p}S_{p}(C_2)$.
\end{theorem}

\section{Lower bounds for energy of gain graphs in terms of minimum vertex degree }\label{lower_bound}
In \cite{Xiaobin-Ma}, Ma established that $\mathcal{E}(G)\geq 2\delta(G)$ for a connected undirected graph $G$ and studied the classes of graphs for which the equality hold, where $ \delta(G) $ is the minimum vertex degree. First, we note that this bound need not be true for arbitrary connected $\mathbb{T}$-gain graphs. 
\begin{example}\label{ex2}{\em
		Consider $ \Phi=(K_3, \varphi) $, where $K_3$ denotes the complete graph on $3$ vertices, with the gains are given by  $
		\varphi(\vec{e}_{1,2})=\varphi(\vec{e}_{2,3})=\varphi(\vec{e}_{3,1})=i$. Then $A(\Phi)= \begin{pmatrix}
			0 & i & -i\\
			-i& 0 & i\\
			i & -i&0
		\end{pmatrix}$. Then $ \spec(\Phi)=\{ -\sqrt{3}, 0, \sqrt{3}\} $. Thus $ \mathcal{E}(\Phi)=2\sqrt{3}\approx 3.464 < 2\delta(\Phi) $.} 
\end{example}

Surprisingly, the counterpart of the result of Ma holds for the triangle-free connected $ \mathbb{T} $-gain graphs,  which is the main result of this section (Theorem \ref{Th3.3}). We shall prove this result with the aid of the following auxiliary results. For a gain graph $\Phi = (G, \varphi)$,  we first establish a lower bound for the energy of a vertex in terms of the vertex degrees and the spectral radius of the graph $G$.
\begin{theorem}\label{Th3.1}
	Let $ \Phi=(G, \varphi) $ be a connected $ \mathbb{T} $-gain graph with at least one edge. Then
	\begin{equation}\label{main-res-1}
		\mathcal{E}_{\Phi}(v_i)\geq \frac{d(v_i)}{\rho(G)},\qquad\text{for all } v_i \in V(G).
	\end{equation}
	Further, equality holds in (\ref{main-res-1}) if and only if $ \Phi $ is balanced and $G$ is complete bipartite.
\end{theorem}

\begin{proof}
	Let $ \lambda_n\leq \lambda_{n-1}\leq \cdots \leq \lambda_1 $ be the eigenvalues of $ \Phi $. Then, by Theorem \ref{Th2.4}, $ \rho(\Phi)\leq \rho(G)$, where $ \rho(\Phi)=\max\{ -\lambda_n, \lambda_1\}$. Therefore $  -\rho(G) \leq \lambda_i \leq \rho(G) $ for all $ i $. Since $ \left| \frac{\lambda_i}{\rho(G)}\right|\leq 1 $, so $ \left( \frac{\lambda_i}{\rho(G)}\right)^{2}\leq \left| \frac{\lambda_i}{\rho(G)}\right|$; and equality holds if and only if $ \lambda_i \in \{ -\rho(G), 0, \rho(G)\} $. Note that, the $i$-th diagonal entry of $A(\Phi)^2$ is $d(v_i)$.  Thus, by  Lemma \ref{Lm2.4}, we get \begin{equation}\label{eq20}
		\mathcal{E}_{\Phi}(v_i) =\sum\limits_{j=1}^{n}|q_{ij}|^{2}|\lambda_j|\geq \sum\limits_{j=1}^{n}|q_{ij}|^{2}\frac{\lambda_j^{2}}{\rho(G)}=\frac{d(v_i)}{\rho(G)},
	\end{equation}
	where $ Q=(q_{ij}) $ is the unitary matrix whose columns are eigenvectors of $ \Phi $. 
	
	For any $ i $, equality holds in the equation \eqref{eq20}  if and only if $ \lambda_j \in \{ -\rho(G), 0, \rho(G)\} $ for all $ j $. As $ G $ has at least one edge, a non-zero eigenvalue exists, say  $ \lambda_i $, different from zero. Thus $ |\lambda_i|=\rho(G)$, and hence $ \rho(\Phi)=\rho(G) $. Therefore, by Theorem \ref{Th2.4}, either $ \Phi $ is balanced or $ -\Phi $ is balanced. 
	
	If $ \Phi $ is balanced, then, by Theorem \ref{Th2.5}, $ \spec(\Phi)=\spec(G) $. As $G$ is connected by the Perron-Frobenius theorem, $\rho(G) $ is a simple eigenvalue of $A(G)$. Since the trace of $ A(G) $ is zero, so $ \spec(\Phi)=\spec(G)=\{ -\rho(G)^{(1)}, 0^{(n-2)}, \rho(G)^{(1)}\} $. Thus $ \Phi=(G, \varphi)$ is a bipartite $ \mathbb{T} $-gain graph with exactly one positive eigenvalue. Therefore, by Theorem \ref{Th2.7}, $ \Phi $ is balanced complete bipartite. 
	
	If $ -\Phi $ is balanced, then $ -\Phi $ is balanced complete bipartite. By Theorem \ref{Th26}, we have $ -\Phi $ is balanced implies $ \Phi $ is balanced. Thus $ \Phi $ is balanced complete bipartite.
	
	Conversely, if $ \Phi $ is  balanced complete bipartite $ \mathbb{T} $-gain graph, then $ \spec(\Phi)=\{ -\rho(G)^{(1)}\break, 0^{(n-2)}, \rho(G)^{(1)}\} $. Hence, equality occurs in \eqref{eq20}.
\end{proof}

Next, we give a lower bound for the energy in terms of the number of edges and the spectral radius of the underlying graph.
\begin{theorem}
	Let $ \Phi=(G, \varphi) $ be a connected $ \mathbb{T} $-gain graph. Then $ \mathcal{E}(\Phi)\geq\frac{2m}{\rho(G)} $, where $m$ is the number of edges in $G$. Equality holds if and only if $ \Phi $ is balanced complete bipartite.
\end{theorem}

\begin{proof}
	Let $ \{ v_1, v_2, \dots, v_n\} $ be the vertex set of $\Phi$.
	Then $ \mathcal{E}(\Phi)=\sum\limits_{j=1}^{n}\mathcal{E}_{\Phi}(v_j) $. By Theorem \ref{Th3.1}, $ \mathcal{E}(\Phi)=\sum\limits_{j=1}^{n}\mathcal{E}_{\Phi}(v_j)\geq \frac{2m}{\rho(G)}$, and the equality holds if and only if $ \Phi $ is balanced complete bipartite.
\end{proof} 	

The following is the main result of this section.

\begin{theorem} \label{Th3.3}
	Let $ \Phi=(G, \varphi) $ be a connected triangle-free $ \mathbb{T} $-gain  graph. Then\break $ \mathcal{E}(\Phi)\geq 2\delta $, where $\delta$ is the minimum vertex degree of $G$. Equality holds if and only if $ \Phi $ is a balanced complete bipartite $ \mathbb{T} $-gain graph and both the partite sets have size $ \delta$.
\end{theorem}
\begin{proof}
	By Theorem \ref{Th3.1}, Theorem \ref{Th2.1}, and Theorem \ref{Th2.2}, we have
	\begin{equation}\label{eq24}
		\mathcal{E}(\Phi)=\sum\limits_{j=1}^{n}\mathcal{E}_{\Phi}(v_j)\geq \sum\limits_{j=1}^{n}\frac{d(v_j)}{\rho(G)}\geq \frac{n\delta}{\rho(G)}\geq \frac{n\delta}{\sqrt{m}}\geq \frac{n\delta}{\sqrt{\lfloor\frac{n^{2}}{4}\rfloor}}\geq \frac{n\delta}{\sqrt{\frac{n^{2}}{4}}}=2\delta.
	\end{equation}
	Equality holds if only if all the terms in the equation \eqref{eq24} are equal. By Theorem \ref{Th3.1},\break $\sum\limits_{j=1}^{n}\mathcal{E}_{\Phi}(v_j) = \sum\limits_{j=1}^{n}\frac{d(v_j)}{\rho(G)}$ if and only if $ \Phi $ is balanced complete bipartite $ \mathbb{T} $-gain graph.  It is easy to see that $\sum\limits_{j=1}^{n}\frac{d(v_j)}{\rho(G)}= \frac{n\delta}{\rho(G)}$ if and only if $ G $ is regular with degree $ \delta$. By Theorem  \ref{Th2.3}(i),\break $\rho(G) = \sqrt m$ if and only if $ G $ is a complete bipartite graph. By Mantel's theorem,\break  $\frac{n\delta}{\sqrt{m}}= \frac{n\delta}{\sqrt{\lfloor\frac{n^{2}}{4}\rfloor}}$ if and only if $G$ is bipartite with equal partition size.  Combining these observations,  we get $ \Phi\sim(K_{\delta,\delta},1) $.  Thus $ \mathcal{E}(\Phi)=2\delta$ implies $ \Phi\sim(K_{\delta,\delta},1) $. The converse is easy to verify.
\end{proof}

\begin{remark}
	Theorem \ref{Th3.3} holds for signed and mixed graphs as well.
\end{remark}
\section{ Upper bound for the energy of gain graphs in terms of the matching number}\label{upper_bound}
In this section, we establish an upper bound for the energy of $ \mathbb{T} $-gain graphs in terms of the matching number of the underlying graph. This result is the counterpart of the corresponding known result for the undirected graphs \cite{ S_Akbari, Pan-Chen-Li, F.Tian-D.Wong}. We characterize all the $ \mathbb{T} $-gain graphs for which the upper bound is attained. We begin this section with needed results, including the Coulson integral formula for the energy of $\mathbb{T}$-gain graphs and a relationship between the characteristic and the matching polynomials of  $ \mathbb{T} $-gain graphs.

The $ \mathbb{T} $-gain energy of $\mathbb{T}$-gain graph $\Phi$ is the sum of the singular values of $A(\Phi)$. This section provides an integral representation for the $\mathbb{T}$-gain energy. Using this representation, we establish a bound for the $\mathbb{T}$-gain energy. Let $ \Phi=(G, \varphi) $ be a $ \mathbb{T} $-gain graph, and let $ P_{\Phi}(x):=\det(xI-A(\Phi))$ be the characteristic polynomial of $\Phi $.

Let $ \Phi $ be a $ \mathbb{T} $-gain graph on $ n $ vertices with the characteristic polynomial $ P_{\Phi}(x) $. Then, by Theorem \ref{Th2.9}, 
\begin{eqnarray}\label{eq6}
	\mathcal{E}(\Phi)=\frac{1}{\pi}\int_{-\infty}^{+\infty}\frac{1}{t^{2}}\log\bigg|t^{n}P_{\Phi}\left(\frac{i}{t}\right)\bigg|dt.
\end{eqnarray}

Let $ P_{\Phi}(x)= x^n+b_1(\Phi)x^{n-1}+b_2(\Phi)x^{n-2}+\cdots+b_n(\Phi)$ be the characteristic polynomial of $ \Phi $. Then  
simplifying the expression \eqref{eq6}, we have

\begin{align}\label{eq4}
	\mathcal{E}(\Phi)&=\frac{1}{\pi}\int_{-\infty}^{+\infty}\frac{1}{t^{2}}\log\bigg |1+ t^{n}\sum\limits_{k=1}^{n}b_{k}(\Phi)\left(\frac{i}{t}\right)^{n-k} \bigg|dt 	\nonumber\\
	&=\frac{1}{\pi}\int_{-\infty}^{+\infty}\frac{1}{t^{2}}\log\bigg |1+t^{n}\sum\limits_{k=1}^{n} (-i)^{k}b_{k}(\Phi)t^{k-n} \bigg|dt \nonumber\\
	&=\frac{1}{2\pi}\int_{-\infty}^{+\infty}\frac{1}{t^{2}}\log\bigg \{A(t)^2+B(t)^2 \bigg \}dt,
\end{align}

where $ A(t)=\bigg(1+\sum\limits_{k=1}^{\lfloor\frac{n}{2} \rfloor} (-1)^kb_{2k}(\Phi)t^{2k} \bigg)$ and $ B(t)=\bigg(\sum\limits_{k=0}^{\lfloor\frac{n-1}{2} \rfloor}(-1)^kb_{2k+1}(\Phi)t^{2k+1}\bigg) $.

\vspace{.5cm}
The above discussion gives a proof for the following theorem. 
\begin{theorem}\label{Lm4.2}
	Let $ \Phi $ be a $ \mathbb{T} $-gain graph with characteristic polynomial $ P_{\Phi}(x)= x^n+b_1(\Phi)x^{n-1}+b_2(\Phi)x^{n-2}+\cdots+b_n(\Phi).$  Then 
	\begin{equation}\label{eq10}
		\mathcal{E}(\Phi)=\frac{1}{2\pi}\int_{-\infty}^{+\infty}\frac{1}{t^{2}}\log\Bigg \{\bigg(1+ \sum\limits_{k=1}^{\lfloor\frac{n}{2} \rfloor} (-1)^kb_{2k}(\Phi)t^{2k}\bigg)^2+\bigg( \sum\limits_{k=0}^{\lfloor\frac{n-1}{2} \rfloor}(-1)^kb_{2k+1}(\Phi)t^{2k+1}\bigg)^2 \Bigg \}dt.
	\end{equation}		
\end{theorem} 

For a  graph $G$, let  $ m(G,j) $ be the number of matchings of size $ j$ in $ G $. The matching polynomial of a graph $G$ on $n$ vertices, denoted by $ m_G(x) $, is defined as follows   
\begin{equation*}
	m_G(x)=\sum\limits_{j=0}^{\lfloor \frac{n}{2}\rfloor} (-1)^jm(G,j)x^{n-2j}.
\end{equation*}
Define $ m(G,0)=1$.  A \textit{matching} in a $ \mathbb{T} $-gain graph $  \Phi=(G, \varphi) $ is the matching in the underlying graph $ G $, and the \textit{matching polynomial} of $ \Phi $, denoted by $ m_\Phi(x) $, is defined as  $ m_\Phi(x)=m_G(x) $.

For a subgraph $K$  of a graph $G$,  let $n(K)$ denote the number of components in $K$, and $G-K$ is the graph obtained from $G$ by deleting the vertices of $K$. 
In \cite{Godsil_Gutman}, Godsil and Gutman established the following relationship between the characteristic polynomial and the matching polynomial of a simple graph $ G $. 

\begin{theorem}[{\cite[Theorem 4]{Godsil_Gutman}}]\label{Th4.1}
	Let $ P_G(x) $ and $ m_{G}(x) $ be the characteristic and the matching polynomials of a simple graph $ G $, respectively. Then 
	\begin{equation*}
		P_G(x)=m_G(x)+\sum\limits_{K}(-2)^{n(K)}m_{G-K}(x),
	\end{equation*}
	where the summation is over all nontrivial subgraphs $ K$ of $ G $, which are the union of vertex disjoint cycles ( i.e.,  $ 2 $-regular subgraphs of $ G $).
\end{theorem} 

Our next objective is to extend the above theorem for the  $ \mathbb{T} $-gain graphs.
\begin{theorem}\label{Th4.2}
	Let $ \Phi=(G, \varphi) $ be a $ \mathbb{T} $-gain graph, and  $ P_\Phi(x) $ and $ m_G(x) $ be the characteristic and the matching polynomials of $\Phi$, respectively. Then 
	\begin{equation*}
		P_\Phi(x)=m_G(x)+\sum\limits_{K}(-2)^{n(K)}\prod\limits_{C\in K}\Re(\varphi(C))m_{G-K}(x),
	\end{equation*}
	where the summation is over all nontrivial subgraphs $ K$ of $ G $ which are union of vertex disjoint cycles, and the product is over all cycles $C$ in $K$.
\end{theorem} 

\begin{proof}
	
	
	Let $ P_\Phi(x) =x^n+\sum\limits_{j=1}^{n}b_j(\Phi)x^{n-j}$	be the characteristic polynomial of $ \Phi $. By Lemma \ref{Lm2.3}, 
	\begin{equation}\label{eq12} P_\Phi(x) =\sum\limits_{j=0}^{n}\left( \sum\limits_{H\in \mathcal{H}_j}(-1)^{n(H)}2^{c(H)}\prod\limits_{C\in \mathcal{C}(H)} \Re(\varphi(C))\right)x^{n-j}.
	\end{equation}
	Note that $ H $ is an arbitrary elementary subgraph of $ G $ with $ j $ number of vertices. Then, each component of $ H $ is either a cycle or an edge. Split $ H $ into two subgraphs $ K $ and $ F $ such that $ K $ contains all the cycles (vertex disjoint) and $ F $ contains the remaining edges (independent edges) of $ H $, respectively. Therefore, $ H= F \cup K $ such that the number of vertices $ j $ in $ H $ can be expressed as $ j=2n(F)+|K| $, where $ |K| $ denotes the number of vertices in $K$. By rearranging the internal summation of \eqref{eq12}  in terms of $K$ and $F$, we get 
	
	\begin{equation*}
		P_\Phi(x) =\sum\limits_{j=0}^{n} \sum\limits_{K} \sum\limits_{F} (-2)^{n(K)} (-1)^{n(F)} \prod\limits_{C\in K} \Re(\varphi(C))x^{n-(2n(F))+|K|)},
	\end{equation*}where $ j=2n(F)+|K| $.
	Now, consider $ K $ as an arbitrary subgraph of $ G $ such that it contains only vertex-disjoint cycles. After reordering the summations, we get 	
	\begin{align*}
		P_\Phi(x) &=\sum\limits_{K}(-2)^{n(K)}\prod\limits_{C\in K} \Re(\varphi(C)) \sum\limits_{n(F)=0}^{\lfloor \frac{n-|K|}{2} \rfloor} \sum\limits_{F}  (-1)^{n(F)} x^{n-(2n(F)+|K|)}\\
		&=\sum\limits_{K}(-2)^{n(K)}\prod\limits_{C\in K} \Re(\varphi(C)) \sum\limits_{n(F)=0}^{\lfloor \frac{n-|K|}{2} \rfloor} m(G-K, n(F))  (-1)^{n(F)} x^{n-(2n(F)+|K|)}\\
		&=\sum\limits_{K}(-2)^{n(K)}\prod\limits_{C\in K} \Re(\varphi(C)) \sum\limits_{t=0}^{\lfloor \frac{n-|K|}{2} \rfloor}(-1)^t m(G-K,t)  x^{n-|K|-2t}\\
		&=\sum\limits_{K}(-2)^{n(K)}\prod\limits_{C\in K} \Re(\varphi(C))m_{G-K}(x).
	\end{align*}
	We split the summation into two parts, $ K=\phi $
	and $ K\ne \phi $. 
	Then 
	\begin{align*}
		P_\Phi(x)= m_G(x)+\sum\limits_{K\ne \phi}(-2)^{n(K)}\prod\limits_{C\in K} \Re(\varphi(C))m_{G-K}(x)
	\end{align*}
	where the summation is over all nontrivial subgraph $ K $ of $ G $ which are union of vertex disjoint cycles, $ n(K) $ is the number of cycles in $ K $, and $ G-K $ is the graph obtained from $ G $ by deleting the vertices of $K $.
\end{proof}

%

\begin{remark}

	The matching polynomial of a simple graph is the same as its characteristic polynomial if and only if the graph is a forest \cite{Godsil_Gutman}. By Theorem  \ref{Th4.2}, it is easy to see that if $ \Phi=(G, \varphi) $ is a $ \mathbb{T} $-gain graph and $G$ is a forest, then the matching polynomial and the characteristic polynomial are the same for $\Phi$. However, in general, the converse need not be true. 
	Let $ \Phi=(C_3, \varphi) $ be a $ \mathbb{T} $-gain graph on a triangle $ C_3:v_1\sim v_2 \sim v_3\sim v_1 $   such that $ \varphi(\vec{e}_{1,2}) =\varphi(\vec{e}_{2,3})=\varphi(\vec{e}_{3,1})=i$. Then $ P_{\Phi}(x)=m_{\Phi}(x)=x^3-3x $. 
\end{remark}
The following result provides a sufficient condition for a $\mathbb{T}$-gain graph with the same characteristic and matching polynomials. 

\begin{corollary}\label{Cor4.2}
	Let $ \Phi=(G, \varphi) $ be a $ \mathbb{T} $-gain graph with $ \Re(\varphi(C))=0 $ for all  cycle $ C $ in $ G $. Then $$ P_\Phi(x)=m_G(x)=\sum\limits_{j=0}^{\lfloor \frac{n}{2} \rfloor} (-1)^{j}m(G,j)x^{n-2j} .$$
\end{corollary}
As a consequence of the above result, we have a formula for energy. 
\begin{corollary}\label{Cor4.3}
	Let $ \Phi=(G, \varphi) $ be a $ \mathbb{T} $-gain graph with $ \Re(\varphi(C))=0 $ for all cycle $ C $ in $ G $. Then 
	\begin{equation*}
		\mathcal{E}(\Phi)=\frac{2}{\pi}\int\limits_{0}^{\infty}\frac{1}{x^2}\log\bigg( 1+\sum\limits_{j=1}^{\lfloor\frac{n}{2} \rfloor}m(G,j)x^{2j}\bigg)dx.
	\end{equation*}
\end{corollary}
\begin{proof}
	Let $ P_{\Phi}(x)=x^n+b_1(\Phi)x^{n-1}+b_2(\Phi)x^{n-2}+\cdots+b_n(\Phi) $ be the characteristic polynomial of $ \Phi $. From Corollary \ref{Cor4.2}, we have $P_\Phi(x)=\sum\limits_{j=0}^{\lfloor \frac{n}{2} \rfloor} (-1)^{j}m(G,j)x^{n-2j}$. Comparing coefficients of both the above polynomials, we get $ b_{2j}(\Phi)=(-1)^jm(G,j) $ and $ b_{2j+1}(\Phi)=0 $, for $ j=0,1, \dots, \lfloor \frac{n}{2} \rfloor $. Now, the result follows from Theorem \ref{Lm4.2}. 
\end{proof}
Corollary \ref{Cor4.3} is useful in constructing equienergetic $ \mathbb{T} $-gain graphs which are not switching equivalent. Two $ \mathbb{T} $-gain graphs $ \Phi_1 $ and $ \Phi_2 $ are \textit{ equienergetic} if $ \mathcal{E}(\Phi_1)=\mathcal{E}(\Phi_2)$.   Let $ G $ be a graph with cycles that are vertex disjoint. 	If $ \Phi_1=(G, \varphi_1)$ and $ \Phi_2=(G, \varphi_2) $ are two $ \mathbb{T} $-gain graphs such that $\Re(\varphi_j(C))=0 $, $ j=1,2 $ for all cycles $ C $ in $ G $, then $\Phi_1$ and $\Phi_2$ are equienergetic. 

%


Define the graphs $G_{p,q,r}$ and $T_{p,q}$, with $p, q, r \geq 0$ as follows: $G_{p,q,r}$ is a graph on $(p+q+r+2)$ vertices with $r$ triangles with a common base edge, say $e = \{u, v\}$, and the vertex $u$ is adjacent to $p$ pendent vertices,  and $v$ is adjacent to $q$ pendent vertices. The tree $T_{p, q}$ consists of $(p+q + 2)$ vertices, where  there is an edge $ e =\{u,v\}$ such that $u$ is adjacent to $p$ pendent vertices, and $v$ is adjacent to $q$ pendent vertices. Both the graphs are depicted in Figure 2.

\begin{figure} [!htb]
	\begin{center}
		\includegraphics[scale= 0.65]{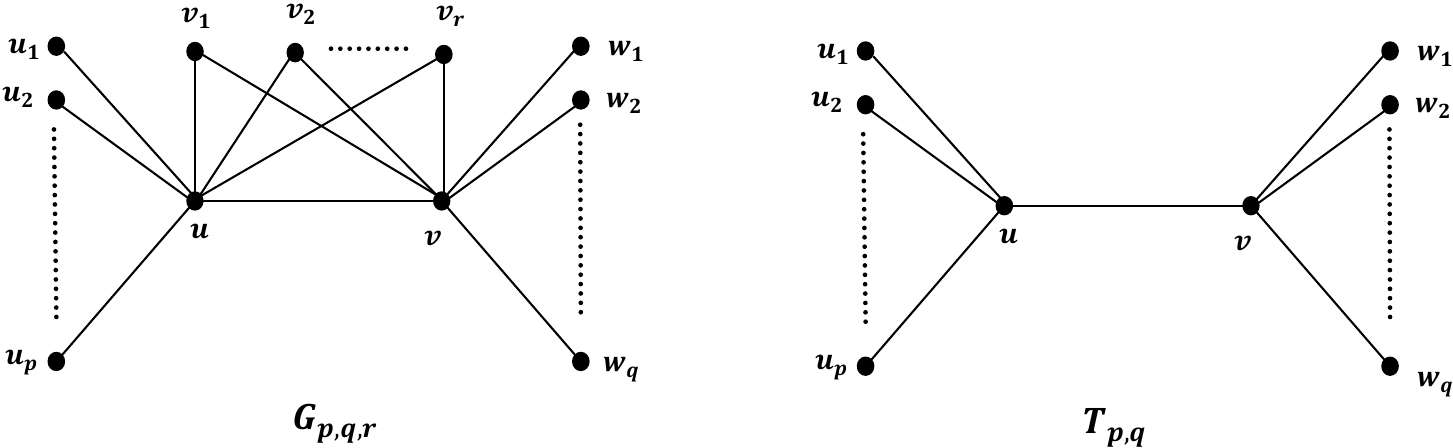}
		\caption{ Graphs $ G_{p,q,r} $ and $ T_{p,q} $} \label{Figg1}
	\end{center}
\end{figure}

Our next objective is to obtain bounds for the energy of a $ \mathbb{T} $-gain graph in terms of the matching number. The corresponding result is known for undirected graphs \cite[Theorem 12]{S_Akbari}. The following key result is used in the proof of \cite[Theorem 12]{S_Akbari}.

\begin{lemma} [{\cite{Pan-Chen-Li}}] \label{lm4.1}For any positive integer $ r $, 
	$ \mathcal{E}(G_{p,q,r})<\mathcal{E}(T_{p+r, q+r}).$
	
\end{lemma}

Unfortunately,  the above result is not true for $ \mathbb{T} $-gain graphs.

\begin{example}\label{deco-coun-exam}
	
	Consider two graphs $ \Phi=(G_{1,1,2}, \varphi) $ and $ T_{3,3} $ given in Figure \ref{fig5}. For the $ \mathbb{T} $-gain graph $ \Phi $, $ \varphi(\vec{e}_{1,2})=\varphi(\vec{e}_{2,1})=-1$, and gain of all other edges are $ 1 $. Then $ \mathcal{E}(\Phi)\not<\mathcal{E}(T_{3,3})$. 
	\begin{figure} [!htb]
		\begin{center}
			\includegraphics[scale= 0.60]{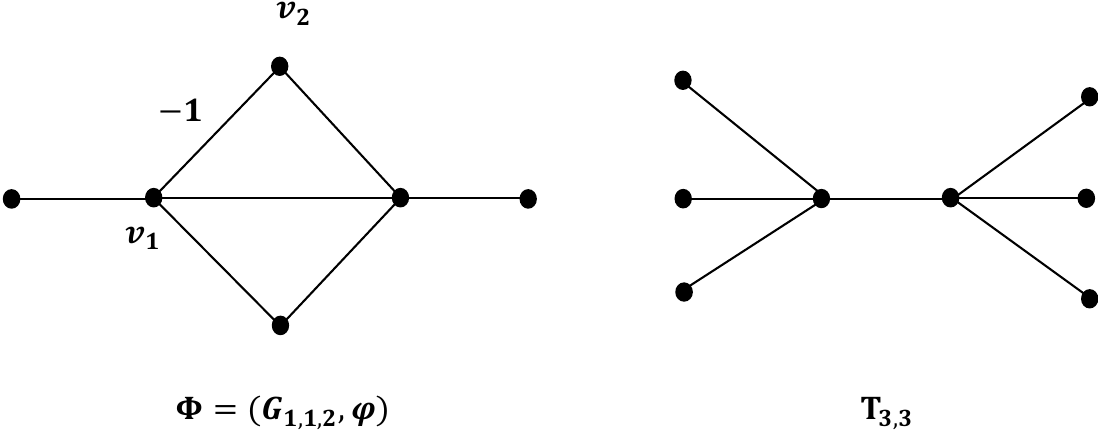}
			\caption{ $ \mathbb{T} $-gain graph $ \Phi=(G_{1,1,2}, \varphi) $ and $T_{3,3}$}
			\label{fig5}
		\end{center}
	\end{figure}
\end{example}

Nevertheless, we have the following result for $ \mathbb{T} $-gain graphs, which is good enough for our purpose. 

\begin{lemma}\label{Lm4.4} Let $ \Phi=(G_{p,q,1}, \varphi) $ be  any $ \mathbb{T} $-gain graph. Then $ \mathcal{E}(\Phi)<\mathcal{E}(T_{p+1, q+1}).$
\end{lemma} 

\begin{proof}
	
	Let  $\Phi=( G_{p,q,1}, 1)$ be a $ \mathbb{T} $-gain graph on the underlying graph $ G_{p,q,1} $ given in Figure 2. Then $\Phi$ contains exactly one cycle of length $3$, say $ C $. Let $ k=p+q+2 $. Then $ k\geq 2 $. Define  $ \Phi_1=\Phi\cup K_1 $. Then $ \mathcal{E}(\Phi)=\mathcal{E}(\Phi_1) $. Here $ |V(\Phi_1)|=|V(T_{p+1,q+1})|=k+2 $. Since $ T_{p+1, q+1} $ is a tree, so 
	$$ P_{T_{p+1, q+1}}(x)=m_{T_{p+1, q+1}}(x)=x^{k+2}-(k+1)x^{k}+(1+p+q+pq)x^{k-2}.$$
	
	\noindent Also, by Theorem \ref{Th4.2}, we have \begin{align*}
		P_{\Phi_1}(x)&=m_{\Phi_1}(x)+ (-2)^1 \Re(\varphi(C))m_{\Phi_1-C}(x)\\
		&=x^{k+2}-(k+1)x^k+(pq+p+q)x^{k-2}-2m_{(k-1)K_1}(x)\Re(\varphi(C))\\
		&=x^{k+2}-(k+1)x^k+(pq+p+q)x^{k-2}-2x^{k-1}\Re(\varphi(C)).\\
	\end{align*}
	
	By Coulson integral formula \eqref{eq10} for the $ \mathbb{T} $-gain graphs, we have 
	
	$$  \mathcal{E}(T_{p+1,q+1})=\frac{1}{2\pi} \int\limits_{-\infty}^{+\infty} \frac{1}{x^2} \log \left( f_1(x)\right)dx$$
	$$  \mathcal{E}(\Phi_1)=\frac{1}{2\pi} \int\limits_{-\infty}^{+\infty} \frac{1}{x^2} \log \left( f_2(x)\right)dx,$$
	where 
	$$ f_1(x)=\left(1+(k+1)x^2+(1+p+q+pq)x^4\right)^2 $$
	and $$ f_{2}(x)=\left(1+(k+1)x^2+(pq+p+q)x^4 \right)^2+\left( 2\Re(\varphi(C)) x^3\right)^2.$$
	Then,
	\begin{align*}
		f_1(x)-f_2(x)&=x^4\{ 2+2(k+1)x^2+2(1+p+q+pq)x^4-x^4\}-4x^6\Re(\varphi(C))^2.
	\end{align*}
	Since  $ |\Re(\varphi(C))| \leq 1$, so 
	\begin{align*}
		f_1(x)-f_2(x)&\geq x^4\{ 2+2(k+1)x^2+(1+2p+2q+2pq)x^4\}-4x^6\\
		&= x^4\{ 2+(2p+2q+2)x^2+(1+2p+2q+2pq)x^4\}>0.
	\end{align*}
	Now $ f_1(x)-f_2(x)=0$ if and only if $ x=0 $. Therefore, for any non-empty graph, \break $ \mathcal{E}(\Phi)<\mathcal{E}(T_{p+1,q+1}).$ 
\end{proof}

\noindent Let us consider the tree $ T_1 $ of diameter $ 3 $ and the maximum edge degree $ \Delta_e $ given in Figure \ref{fig3} (see Section 2). Now, we compare the energies of  $ (G_{p,q,1},\varphi) $ and $ T_1 $, where $ p,q \geq 0 $. 

\begin{lemma}\label{Lm4.5}
	
	Let $ \Phi=(G_{p,q,1}, \varphi) $ be a $ \mathbb{T} $-gain graph, and $ T_1 $ be a tree given in Figure \ref{fig3} such that the maximum edge degree of $ \Phi $ is less than or equal to the maximum edge degree of $ T_1 $. Then $$ \mathcal{E}(\Phi)< \mathcal{E}(T_1).$$
\end{lemma}
\begin{proof}
	Let $ \Phi=(G_{p,q,1}, \varphi) $ be a $ \mathbb{T} $-gain graph on $ G_{p,q,1} $. By Lemma \ref{Lm4.4}, we have $  \mathcal{E}(\Phi)<\mathcal{E}(T_{p+1, q+1}) $.  Note that, by Lemma \ref{Lm2.1}, $ \mathcal{E}(T_{p+1, q+1})\leq \mathcal{E}(T_{\lfloor \frac{p+q+2}{2} \rfloor, \lceil \frac{p+q+2}{2} \rceil}) $. Consider the tree $ T_1 $ given in Figure \ref{fig3}, where $ \Delta_e (T_1)$ is the maximum edge degree of $ T_1 $. Since\break $ \Delta_e (T_1) \geq p+q+2$, so $ T_{\lfloor \frac{p+q+2}{2} \rfloor, \lceil \frac{p+q+2}{2} \rceil} $ is an induced subgraph of $ T_1 $. Therefore, by Theorem \ref{Th2.10}, $ \mathcal{E}(T_{\lfloor \frac{p+q+2}{2} \rfloor, \lceil \frac{p+q+2}{2} \rceil}) \leq \mathcal{E}(T_1) $. By combining the above three inequalities, we have $ \mathcal{E}(\Phi)<\mathcal{E}(T_1) $.
\end{proof}
Next, we compare the energies of  $ (T_{p,q},\varphi) $ and $ T_1 $, where $ p,q \geq 0 $. 
\begin{lemma}\label{Lm4.6}	
	Let $ \Phi=(T_{p,q}, \varphi) $ be a $ \mathbb{T} $-gain graph and $ T_1 $ be a tree given in Figure \ref{fig3} such that the maximum edge degree of $ \Phi $ is less than or equal to the maximum edge degree of $ T_1 $. Then $ \mathcal{E}(\Phi) \leq \mathcal{E}(T_1) $. Equality holds if and only if $ T_{p,q}$ is isomorphic to $T_1 $.
\end{lemma}
\begin{proof}
	By Lemma \ref{Lm2.1}, $ \mathcal{E}(\Phi)=\mathcal{E}(T_{p,q})\leq \mathcal{E}(T_{\lfloor \frac{p+q}{2} \rfloor, \lceil \frac{p+q}{2} \rceil}) $. Let $ \Delta_e $ be the maximum edge degree of $ T_1 $. Since $ p+q\leq \Delta_e $, so $ T_{\lfloor \frac{p+q}{2} \rfloor, \lceil \frac{p+q}{2} \rceil} $ is an induced subgraph of $ T_1 $. Therefore, by Theorem \ref{Th2.10}, $ \mathcal{E}(T_{\lfloor \frac{p+q}{2} \rfloor, \lceil \frac{p+q}{2} \rceil}) \leq \mathcal{E}(T_1) $. Thus, $ \mathcal{E}(\Phi)\leq \mathcal{E}(T_1) $. If $ \mathcal{E}(\Phi)=\mathcal{E}(T_1) $, then all the intermediate inequalities will be equal. Therefore, by Theorem \ref{Th2.10} and Lemma \ref{Lm2.1}, $ T_{p,q}$ is isomorphic to $T_1$. 
\end{proof}

In \cite{S_Akbari}, Akbari et al. decomposed an undirected graph $ G $, with matching number $\mu$, into $ \mu $ copies edge-disjoint copies of  $ G_{p,q, r} $, where $ r \geq 0 $. In the next lemma, we decompose $ G $ into $ \mu $ edge-disjoint copies of either $ G_{p,q,0} $ or $ G_{p,q,1} $. This result is crucial in proving the main result of this section.
\begin{lemma}\label{Lm4.7}
	Any graph $ G $ with matching number $ \mu $ can be decomposed into $ \mu $  edge-disjoint copies either $ T_{p,q} $ or $ G_{p,q,1} $.
\end{lemma}

\begin{proof}
	Let $ G $ be a graph with matching number $\mu $. Let $ \mathcal{M}=\{ e_1, e_2, \dots, e_{\mu}\} $ be a maximum matching of $ G $.  Let  $ H_j $ be a subgraph of $ G $ containing $ e_j $ and its all incidence edges, for $ j=1,2, \cdots, \mu $. Take $ S_1=H_1 $ and $ S_j=H_j\setminus \bigcup\limits_{i=1}^{j-1} E(H_i) $, for $ j=2,3, \dots, \mu $. Note that each $ S_j $ is of the form $ G_{p,q,r} $, where $ r\geq 0 $.
	
	For $ j=1,2, \dots, \mu $, construct the subgraph $ L_j $ from the graph $ S_j$ as follows: Suppose two of the nonmatching edges of any triangle in $S_j$ are incident with a matching edge $e_i$, for $i \neq j$, then remove the edges from $S_j$ and add them to $S_i$. That is, if the union $ S_j\cup S_i $ contains the subgraph $ K $ ( see Figure \ref{fig6}), then keep the edges $ e_{x,t} $ and $ e_{y,t} $ in the subgraph $ S_i $, and discard these edges in $ S_j $.  
	Continue this process until none of the modified subgraphs contain a triangle with two of its edges incident with two matching edges in the subgraph. For $ j=1,2 \dots, \mu $, let $ L_j $ denote the subgraph containing the matching edge $e_j$ obtained at the end of the above process. It is easy to see that all subgraphs $ L_j $ are edge-disjoint, and their union is $ G $.
	
	\begin{figure} [!htb]
		\begin{center}
			\includegraphics[scale= 0.60]{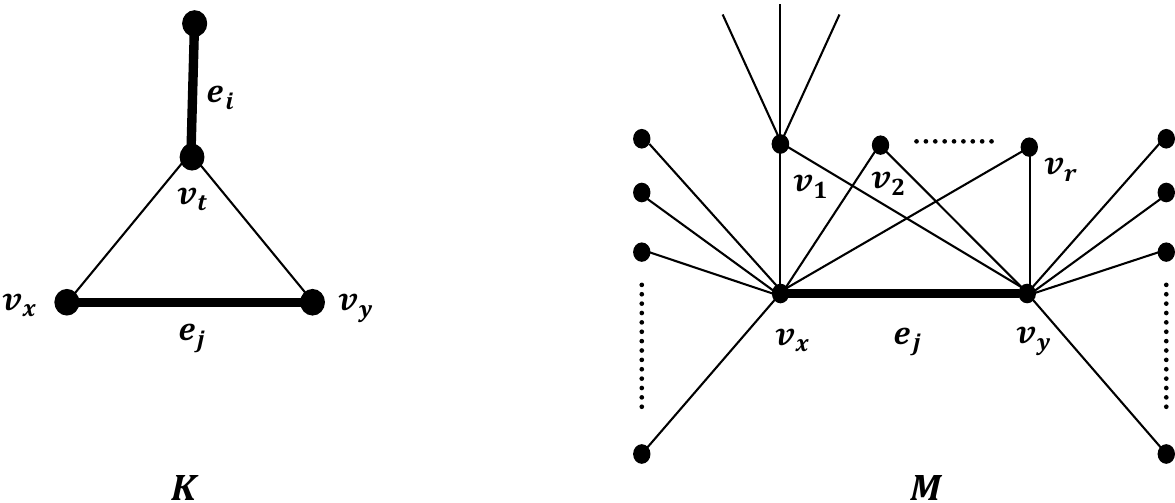}
			\caption{}
			\label{fig6}
		\end{center}
	\end{figure}
	
	Next, we claim that, each $ L_j $ is of the form either $ T_{p,q} $ or $ G_{p,q,1} $. For, suppose that $ L_j $ contains more than one triangle. Let $ v_1,v_2, \dots, v_r $ be the vertices of the triangles (see graph $ M $ in Figure \ref{fig6}), where $ r\geq 2 $. From the construction of $ L_j $,  none of the vertices $ v_1, v_2, \dots, v_r $ is adjacent to any matching edges in $ \mathcal{M} $. Therefore,  the collection $ \mathcal{M'}=\{ e_1,e_2, \dots, e_{j-1}, e_{j+1}, \dots, e_\mu, e_{x,1}, e_{y,2}\} $ is a matching, and it contains more edges than  $\mathcal{M}$, which is a contradiction. Thus, the result follows.
\end{proof}

Now, we are ready to prove the main result of this section, which gives a bound for the energy in terms of matching number and maximum edge degree. Let $P_n$ denote the path on $n$ vertices.
%

\begin{theorem}\label{Th4.3}
	Let $ \Phi=(G,\varphi) $ be any $\mathbb{T}$-gain graph on a simple graph $ G $ with the matching number $ \mu$ and the maximum edge degree $ \Delta_e$. Then we have the following:
	
	\begin{itemize}
		\item[(A)] If $ \Delta_e$ is even, then $ \mathcal{E}(\Phi)\leq 2\mu\sqrt{2\Delta_e+1} $. Equality holds if and only if $ \Phi $ is the disjoint union of $ \mu$ copies of $ (P_2, \varphi) $ and some isolated vertices.
		\item [(B)]  If $ \Delta_e$ is odd, then $ \mathcal{E}(\Phi)\leq \mu\left( \sqrt{b+2\sqrt{b}}+\sqrt{b-2\sqrt{b}}\right) $, where $ b=2(\Delta_e+1) $. Equality holds if and only if $ \Phi$ is $ \mu$ copies of $ (P_3, \varphi) $ and some isolated vertices.
	\end{itemize} 
\end{theorem}

\begin{proof}
	Let $ \Phi=(G, \varphi) $ be a $ \mathbb{T} $-gain graph  with matching number $ \mu $. Let $ \mathcal{M}=\{ e_1, e_2, \cdots, e_{\mu} \} $ be a maximum matching of $ \Phi $, and let $ d(e_j) $  denote the edge degree of $ e_j $ in $ \Phi $ for $ j=1,2, \cdots,\mu$. Then, by Lemma \ref{Lm4.7}, we can decompose $ \Phi $ into edge-disjoint subgraphs $ (L_j, \varphi) $, for $ j=1,2, \dots, \mu $. As $ A(\Phi)=\sum\limits_{j=1}^{\mu}A(L_j, \varphi) $, by  Theorem \ref{Th2.8}, we get $ \mathcal{E}(\Phi)\leq \sum\limits_{j=1}^{\mu}\mathcal{E}(L_j, \varphi) $. Now each subgraph $ L_j $ is of the form $ G_{p_j, q_j,r}$, with either $ r=0 $ or $ r=1 $, and $ p_j+q_j+2r\leq d(e_j)\leq \Delta_{e} $. Now, by Lemma \ref{Lm4.5} and Lemma \ref{Lm4.6}, $ \mathcal{E}(G_{p_j,q_j, r}, \varphi) \leq \mathcal{E}(T_1)$, for $ j=1, 2, \cdots, \mu $, where $T_1$ is given in Figure \ref{fig3}. Thus $ \mathcal{E}(\Phi)\leq \mu \mathcal{E}(T_1) $. Hence, by Lemma \ref{Lm2.2}, the inequalities in $ (A) $ and $ (B) $ follows. 
	
	Suppose $ \mathcal{E}(\Phi)=\mu \mathcal{E}(T_1) $. Then $ \mathcal{E}(L_j, \varphi)=\mathcal{E}(T_1) $, for $ j=1,2, \cdots, \mu$. Also, by Lemma \ref{Lm4.5}, $ L_j $ can not be the form of $ G_{p_j,q_j,1} $. Thus each $ L_j $ is of the from $ T_{p_j,q_j} $ for $ j=1,2,\cdots, \mu $. By Lemma \ref{Lm4.6}, $ L_j=T_1 $, for $ j=1,2, \dots, \mu $. Each subgraph $ L_j $ is a component of $ G $. For, suppose that $ L_j $ is not a connected component of $ G $. Then there exists an edge $ \bar{e}$ incident with both the matching edges $ e_j $ and $ e_t $ for some $ t\ne j$, otherwise we can find a matching larger than $ \mu $. Without loss of generality, let us assume that $ \bar{e} $ is an edge in $ L_t $. The degree of $ e_j $ in $ L_j $ is less than or equal to $ d(e_j)-1 $. Since $ L_j=T_1 $, so the degree of $ e_j $ in $ L_j $ is $ \Delta_e $. Therefore, $ \Delta_e\leq d(e_j)-1\leq \Delta_e-1 $, which is a contradiction. So $ L_j $ is a connected component. Therefore all the subgraphs $ L_j $, for $ j=1,2, \dots, \mu $ are connected components of $ G $ with maximum edge degree $ \Delta_e $. Suppose that  $ \Delta_{e}\geq 2 $ for some matching edge $e$. Then, by removing the edge $e$ from $\mathcal{M}$ and adding the two edges that are incident with the end vertices of $e$, we get a matching size larger than $ \mu $; this is not possible, as $\mathcal{M}$ is a maximum matching. Thus $ \Delta_e\leq 1 $. If $ \Delta_e=1 $, then $ T_1=P_3 $ and if $ \Delta_e=0 $, then $ T_1=P_2 $. Hence, the result follows.
\end{proof}

\begin{corollary}\label{Cor4.5}
	Let $ \Phi=(G,\varphi) $ be any $\mathbb{T}$-gain graph on a simple graph $ G $ with the matching number $ \mu$ and the maximum edge degree $ \Delta_e$. Then 
	\begin{equation*}
		\mathcal{E}(\Phi)\leq 2\mu\sqrt{2\Delta_e+1}.
	\end{equation*}
\end{corollary}
\begin{proof}
	Note that, if $ c=\sqrt{b+2\sqrt{b}} +\sqrt{b-2\sqrt{b}}>0$, where $ b=2(\Delta_{e}+1) $, then $ c^2=2b+2\sqrt{(b-2)^2-4}<4(b-1) $. Therefore, $ \sqrt{b+2\sqrt{b}} +\sqrt{b-2\sqrt{b}}<2\sqrt{2\Delta_{e}+1} $.
\end{proof}

\begin{remark}
	Hermitian adjacency matrix of a mixed graph and adjacency matrix of a signed graph are particular cases of the adjacency matrix of a $ \mathbb{T} $-gain graph. Therefore, all the above results also hold for mixed and signed graphs.
\end{remark}
	\section{Extremal families for the energy of unicyclic $ \mathbb{T} $-gain graphs}\label{extremum_energy}

For any simple graph $ G $,  let  $ \mathcal{T}_G $  denote the collection of all  $ \mathbb{T} $-gain graphs on the underlying graph $ G $. This section studies the extremum problem for the energies of $ \mathbb{T} $-gain graphs in $ \mathcal{T}_G $. In \cite{Haemers},  Haemers conjectured that among all the signed graphs defined of $K_n$, the complete graph on $n$ vertices, all one edge weights has the least energy. That is,   $ \mathcal{E}(K_{n})\leq \mathcal{E}(K_{n}, \psi) $ for any signed graph $ (K_{n}, \psi) $. In \cite{Akbari}, Akbari et al. proved this conjecture. In this section, first, we observe that the counterpart of this result for $ \mathbb{T} $-gain graphs need not be true. 

%

\begin{example}{\em
		Let $ \Phi=(K_3, \varphi) $ be a $ \mathbb{T} $-gain graph on $ K_3 $ with vertex set $ \{ v_1,v_2,v_3\} $ such that $ \varphi(\vec{e}_{1,2})=i $, $ \varphi(\vec{e}_{2,3})=\varphi(\vec{e}_{3,1})=1 $. Then $ A(\Phi)=\begin{pmatrix}
			0 & i & 1\\
			-i & 0 & 1\\
			1& 1 & 0
		\end{pmatrix} $. Thus $ \mathcal{E}(\Phi)=2\sqrt{3}\approx 3.464 $. Also $ \mathcal{E}(K_3)=4 $. So $ \mathcal{E}(K_3)>\mathcal{E}(K_3, \varphi) $.}
\end{example}

From the following example one can see that  for a graph $ G $  the energy $ \mathcal{E}(G) $ need not be the minimum/maximum energy among  all $ \mathbb{T} $-gain graphs defined on $G$.
\begin{example}{\em
		In Figure \ref{fig1.1}, we can see that $ \Phi_{1} $ and $ \Phi_{2} $ are two $ \mathbb{T} $-gain graphs on a same underlying graph $ G $, where $ \mathcal{E}(\Phi_{1})=6.133, \mathcal{E}(\Phi_{2})=5.9939$ and $ \mathcal{E}(G)=6.0409 $. Thus, $ \mathcal{E}(\Phi_2)<\mathcal{E}(G)< \mathcal{E}(\Phi_{1}) $.}
\end{example}

\begin{figure} [!htb]
	\begin{center}
		\includegraphics[scale= 0.60]{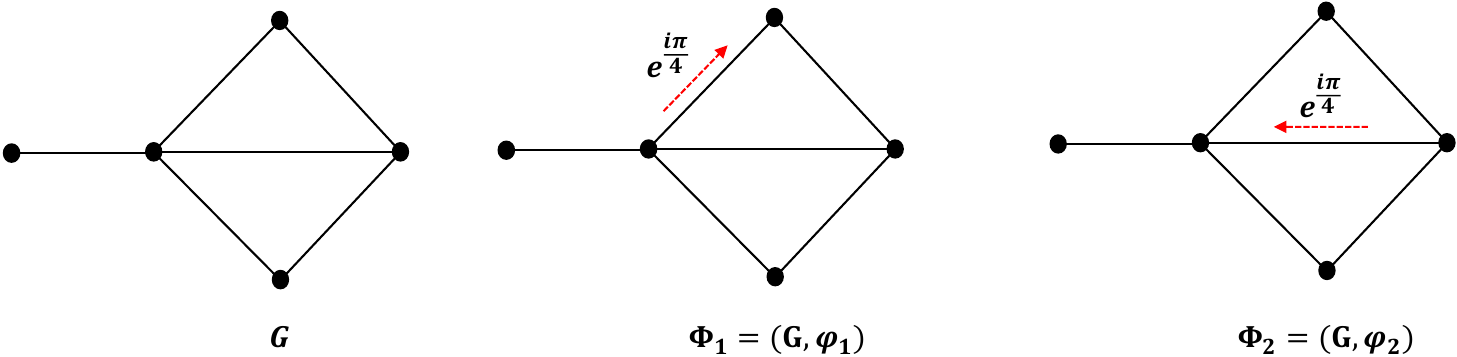}
		\caption{ The gain graphs $ \Phi_{1}$ and $ \Phi_{2}$ have the same underlying graph $ G $.} \label{fig1.1}
	\end{center}
\end{figure}

Motivated by this observation, it is natural to study the graphs for which the energy of the underlying graph attains either maximum or minimum energy among the energies of all $ \mathbb{T} $-gain graphs defined on $ G $. In \cite{our-paper3}, we proved that $ \mathcal{E}(K_{n,n})\leq \mathcal{E}(K_{n,n}, \varphi) $, where $ (K_{n,n}, \varphi) $
is any $ \mathbb{T} $-gain graph on $ K_{n,n} $. Here, we present another class of graphs, namely unicyclic graphs, for which its energy attains extremum energy overall $ \mathbb{T} $-gain graphs associated with it. The girth of a graph $G$ is the length of a shortest cycle in $G$.

\begin{theorem}\label{Th5.1}
	Let $ \Phi=(G, \varphi) $ be a $ \mathbb{T} $-gain graph on a unicyclic graph $ G $ with girth $ r $.  
	\begin{enumerate}
		\item[(i)] If $ r\equiv 0 (\mod  4 )$, then $ \mathcal{E}(\Phi) \geq \mathcal{E}(G) $. Further, equality holds if and only if $ \Phi $ is balanced.
		\item[(ii)]  If $ r \equiv 1 (\mod  4 )$ or $ r\equiv 3  (\mod  4 )$, then $  \mathcal{E}(\Phi) \leq \mathcal{E}(G) $. Further, equality holds if and only if either $ \Phi $ or $ -\Phi $ is balanced.
		\item[(iii)] If $ r \equiv 2 (\mod  4 )$, then $  \mathcal{E}(\Phi) \leq \mathcal{E}(G) $. Further, equality holds if and only if  $ \Phi $  is balanced.
	\end{enumerate}
\end{theorem} 	

\begin{proof}
	Let  $ C_r $ be the unique cycle in $ G $. If $ P_{\Phi}(x)=x^n+b_1(\Phi)x^{n-1}+b_2(\Phi)x^{n-2}+\cdots+b_n(\Phi)$ is the characteristic polynomial of $ \Phi $, then, by Lemma \ref{Lm2.3}, the coefficients are given by:   $$ b_j(\Phi)=\sum\limits_{H\in \mathcal{H}_j}(-1)^{n(H)}2^{c(H)}\prod\limits_{C\in C(H)}\Re(\varphi(C)),~~~j=1,2,\cdots,n.$$ Now, we consider the following cases.
	
	\noindent\textbf{Case (i):} Let $ r\equiv 0 $ (mod $ 4 $).	Since $ r$ is even, so $G$ is bipartite. Therefore, for any $\Phi = (G, \varphi)$, by Theorem \ref{Th2.11}, the eigenvalues of $ \Phi $ are symmetric about the origin. Then $ b_{2k+1}(\Phi)= b_{2k+1}(G)=0$  for $ k=0,1, \cdots, \lfloor \frac{n-1}{2}\rfloor $. Also  $ (-1)^kb_{2k}(\Phi)\geq0 $ and $ (-1)^kb_{2k}(G)\geq0 $ for $ k=1,2,\cdots, \lfloor \frac{n}{2} \rfloor $. Let $ r=4s $ for some positive integer $ s $. Then,  $$ b_{2k}(\Phi)=\left\{ \begin{array}{ccc}
		(-1)^km(\Phi,k) & \mbox{if}& 2k<r\\
		(-1)^km(\Phi,k)+(-1)^{k-2s+1}2\Re(\varphi(C_r))m(\Phi-C_r, k-2s) & \mbox{if} & 2k\geq r,
	\end{array}\right . $$for  $ k=1,2,\cdots, \lfloor \frac{n}{2} \rfloor .$ 	Thus $ (-1)^{k}b_{2k}(\Phi)=(-1)^{k}b_{2k}(G)=m(\Phi,k)\geq 0 $ for $ 2k<r $.	Also, $ (-1)^kb_{2k}(\Phi)=m(\Phi,k)-2\Re(\varphi(C_r))m(\Phi-C_r,k-2s) $ and  $(-1)^kb_{2k}(G)=m(\Phi,k)-2m(\Phi-C_r,k-2s)$, for $ 2k\geq r $. Now $ \Re(\varphi(C_r))\leq 1 $, and equality holds if and only if $ \Phi $ is balanced. Then $ (-1)^kb_{2k}(\Phi)\geq (-1)^kb_{2k}(G) $, for $ 2k\geq r $ and equality holds if and only if $ \Phi $ is balanced. Thus, by Lemma \ref{Lm4.2}, $ \mathcal{E}(\Phi) \geq \mathcal{E}(G) $ and equality holds if and only if $ \Phi $ is balanced.
	
	\vspace{0.4cm}
	\noindent\textbf{Case (ii):} Either $ r \equiv 1$ (mod $ 4 $) or $ r\equiv 3 $ (mod $ 4 $).

	\noindent\textbf{Subcase (a):} Let $ r \equiv 1 $ (mod $ 4 $). Then $ r=4s+1 $ for some nonnegative integer $ s $. Note that $ b_{2k}(\Phi)=(-1)^km(\Phi,k)=b_{2k}(G) $ for $ k=1,2,\cdots, \lfloor \frac{n}{2} \rfloor $. Now, for  $ k=0,1,\cdots, \lfloor \frac{n-1}{2} \rfloor $, \begin{equation*}
		b_{2k+1}(\Phi)=\left\{ \begin{array}{ccc}
			0 & \mbox{if}& 2k+1<r\\
			(-1)^{k-2s+1}2\Re(\varphi(C_r))m(\Phi-C_r, k-2s) & \mbox{if} & 2k+1\geq r
		\end{array}\right.
	\end{equation*}
	That is,
	\begin{equation}\label{eq14}
		(-1)^k	b_{2k+1}(\Phi)=\left\{ \begin{array}{ccc}
			0 & \mbox{if}& 2k+1<r\\
			-2\Re(\varphi(C_r))m(\Phi-C_r, k-2s) & \mbox{if} & 2k+1\geq r
		\end{array}\right.
	\end{equation}
	
	Consider $ -G $ as a $ \mathbb{T} $-gain graph on $ G $ with the gain of each edge is $ -1 $. Note that,  $\mathcal{E}(G)=\mathcal{E}(-G)$. Using \eqref{eq14}, we have 
	\begin{equation}\label{eq15}
		(-1)^k	b_{2k+1}(-G)=\left\{ \begin{array}{ccc}
			0 & \mbox{if}& 2k+1<r\\
			2m(\Phi-C_r, k-2s) & \mbox{if} & 2k+1\geq r
		\end{array}\right.
	\end{equation}
	It is easy to see that $(-1)^k	b_{2k+1}(-G)=-(-1)^k	b_{2k+1}(G)\geq 0  $. Since $ |\Re(\varphi(C_r))|\leq 1 $, by comparing \eqref{eq14} and \eqref{eq15}, we have $ |(-1)^k	b_{2k+1}(\Phi)|\leq (-1)^k	b_{2k+1}(-G) $ for $ k=0,1,\cdots, \lfloor\frac{n-1}{2} \rfloor $. Equality occur if and only if either $ \Re(\varphi(C_r))=1 $ or $ \Re(\varphi(C_r))=-1 $. That is, equality holds if and only if either $ \Phi $ is balanced or $ -\Phi $ is balanced. Therefore, by Lemma \ref{Lm4.2}, $ \mathcal{E}(\Phi)\leq \mathcal{E}(G) $ and equality holds if and only if either $ \Phi $ is balanced or $ -\Phi $ is balanced.  
	
	\noindent\textbf{Subcase (b):} Let $ r \equiv 3 $ (mod $ 4 $). Let $ r=4s+3 $ for some nonnegative integer $s $. Then similar to Subcase (b), $(-1)^k b_{2k}(\Phi)=m(\Phi,k)=(-1)^kb_{2k}(G) $, for $ k=1,2,\cdots, \lfloor \frac{n}{2} \rfloor $ and  $ |(-1)^k	b_{2k+1}(\Phi)|\leq (-1)^k	b_{2k+1}(G) $, for $ k=0,1,\cdots, \lfloor\frac{n-1}{2} \rfloor $. Hence, by Lemma \ref{Lm4.2}, $ \mathcal{E}(\Phi)\leq \mathcal{E}(G) $. Further equality holds if and only if either $ \Phi $ is balanced or $ -\Phi $ is balanced.

	\noindent\textbf{Case(iii):} Let $ r \equiv 2 $ (mod $ 4 $). Since $ r $ is even,  $ b_{2k+1}(\Phi)= b_{2k+1}(G)=0$  for $ k=0,1, \cdots, \lfloor \frac{n-1}{2}\rfloor $. Also  $ (-1)^kb_{2k}(\Phi)\geq0 $ and $ (-1)^kb_{2k}(G)\geq0 $ for $ k=1,2,\cdots, \lfloor \frac{n}{2} \rfloor $. Let $ r=4s+2 $ for some nonnegative integer $ s $. Then similar to previous case, $ (-1)^{k}b_{2k}(\Phi)=(-1)^{k}b_{2k}(G)$ for $ 2k<r $. Also $ (-1)^kb_{2k}(\Phi)\leq (-1)^kb_{2k}(G) $ for $ 2k\geq r $ and equality holds if and only if $ \Phi $ is balanced. Therefore, by Lemma \ref{Lm4.2}, $ \mathcal{E}(\Phi)\leq \mathcal{E}(G) $, and  $ \mathcal{E}(\Phi) = \mathcal{E}(G) $ holds if and only if $ \Phi $ is balanced.
	
\end{proof}
We present an immediate consequence of the above result.

\begin{corollary}\label{Cor5.1}
	Let $ \Phi=(C_{n}, \varphi)$ be any $\mathbb{T}$-gain graph on a cycle  $ C_n $ of $ n $ vertices. Then
	\begin{itemize}
		\item[(i)]If $ n \equiv 0$ ($\mod  4 $), then $ \mathcal{E}(\Phi) \geq \mathcal{E}(C_{n}) $. Equality holds if and only if $ \Phi $ is balanced.
		\item[(ii)] If $ n \equiv 1$ ($\mod  4 $) or $ n\equiv 3 $ ($\mod  4 $), then $  \mathcal{E}(\Phi) \leq \mathcal{E}(C_{n}) $. Equality holds if and only if either $ \Phi $ or $ -\Phi $ is balanced.
		\item[(iii)] If $ n \equiv 2$ ($\mod  4 $), then $  \mathcal{E}(\Phi) \leq \mathcal{E}(C_{n}) $. Equality holds if and only if  $ \Phi $  is balanced.
	\end{itemize}
\end{corollary}

For a simple graph $ G $, the collection $ \mathcal{T}_G $ is called equienergetic if the energy of all $ \mathbb{T} $-gain graphs in $\mathcal{T}_G $ are the same. If $ T $ is a tree, then $ \mathcal{T}_T $ is equienergetic. Next, we consider the problem of characterizing graphs $G$ for which the family $\mathcal{T}_G $ is equienergetic.

\begin{theorem}\label{thm-equi-ener}
	If $ G $ is a graph with all its cycles are vertex disjoint, then $ \mathcal{T}_{G} $ is not equienergetic.
\end{theorem} 
\begin{proof}
	Let $ G $ be a graph with vertex disjoint cycles, and let $ C_1,C_2, \dots, C_p $ be the only cycles in $ G $. Let  $ \Phi_{1} $ and $ \Phi_{2} $  be two $\mathbb{T} $-gain graphs defined on $ G $. Let $ P_{\Phi_j}(x)= x^n+b_1(\Phi_j)x^{n-1}+b_2(\Phi_j)x^{n-2}+\cdots+b_n(\Phi_j)$ be the characteristic polynomial of $ \Phi_j $, for $ j=1,2 $. By Lemma \ref{Lm4.2}, 
	\begin{equation*}
		\mathcal{E}(\Phi_j)=\frac{1}{2\pi}\int_{-\infty}^{+\infty}\frac{1}{t^{2}}\log\Bigg \{A_j(t)^2+B_j(t)^2 \Bigg \}dt,
	\end{equation*}	
	where $ A_j(t)=\bigg(1+ \sum\limits_{k=1}^{\lfloor\frac{n}{2} \rfloor} (-1)^kb_{2k}(\Phi_j)t^{2k}\bigg), B_j(t)=\bigg( \sum\limits_{k=0}^{\lfloor\frac{n-1}{2} \rfloor}(-1)^kb_{2k+1}(\Phi_j)t^{2k+1}\bigg)$ for $ j=1,2 $, and $ b_k(\Phi_j)'s$ for $ k=1,2, \cdots,n$ are given by \eqref{eq5}.
	Now, we consider the following three cases. In each of the cases, we find the choices for the gains $\Phi_1$ and $\Phi_2$ that serve the purpose: 
	
	\noindent{\bf Case (a):} Suppose every cycle in $G$ has odd length.  Define $ \Phi_{1} $ and $ \Phi_{2} $ as follows: $ \varphi_{1}(\vec{C_{1}})=-\frac{1}{2}+i\frac{\sqrt{3}}{2} $, $ \varphi_{1}(\vec{C_k})=i $ for $ k=2,3,\dots, p $; and $ \varphi_{2}(\vec{C_{1}})=-1 $, $ \varphi_{2}(\vec{C_k})=i $ for $ k=2,3,\dots, p.$ Then, $ A_1(t)^2=A_2(t)^2$ and $ B_1(t)^2 \ne B_2(t)^2 $. Thus, $ \mathcal{E}(\Phi_{1}) \ne \mathcal{E}(\Phi_{2})$.
	
	\noindent	{\bf Case (b):} Suppose every cycle in $G$ has even length.  Define $ \Phi_{1} $ and $ \Phi_{2} $ as follows: $ \varphi_{1}(\vec{C_{1}})=-\frac{1}{2}+i\frac{\sqrt{3}}{2} $, $ \varphi_{1}(\vec{C_k})=i $ for $ k=2,3,\dots, p $. Also $ \varphi_{2}(\vec{C_{1}})=-1 $, $ \varphi_{2}(\vec{C_k})=i $ for $ k=2,3,\dots, p.$ Then, $ B(t)=B^{'}(t)=0 $ and $ A_1(t)^2 \ne A_2(t)^2 $. Thus,  $ \mathcal{E}(\Phi_{1})\ne \mathcal{E}(\Phi_{2}) $.
	
	\noindent 	{\bf Case (c):} Suppose $G$ contains both the odd and even cycles. Let $ C_1 $ be an even cycle. Define $ \Phi_{1} $ and $ \Phi_{2} $ as follows:  $ \varphi_{1}(\vec{C})=\varphi_{2}(\vec{C})=i $ for all cycles $ C $  except $ C_1$ and define $\varphi_{1}(\vec{C_1})=-\frac{1}{2}+i\frac{\sqrt{3}}{2} $, $ \varphi_{2}(\vec{C_{1}})=-1 $. The rest of the proof of this case is similar to that of Case (b).
\end{proof}

Given a graph $G$, in  \cite{Our-paper-1}, we established that the spectrum is the same for all the elements of  $\mathcal{T}_G$ if and only if $G$ is a tree. Based on the above results, we conjecture the following: 
\begin{conj}
	For a connected simple graph $ G $, if $ \mathcal{T}_{G} $ is equienergetic then $ G $ is a tree.
\end{conj}
\section*{Acknowledgments}
Aniruddha Samanta thanks the National Board for Higher Mathematics (NBHM), Department of Atomic Energy, India, for financial support in the form of an NBHM Post-doctoral Fellowship (Sanction Order No. 0204/21/2023/R$\&$D-II/10038). M. Rajesh Kannan would like to thank the SRC, IIT Hyderabad for financial support.

	\bibliographystyle{amsplain}
	\bibliography{raj-ani-ref1}

\providecommand{\bysame}{\leavevmode\hbox to3em{\hrulefill}\thinspace}
\providecommand{\MR}{\relax\ifhmode\unskip\space\fi MR }
\providecommand{\MRhref}[2]{%
  \href{http://www.ams.org/mathscinet-getitem?mr=#1}{#2}
}
\providecommand{\href}[2]{#2}
\begin{thebibliography}{10}

\bibitem{Akbari}
S.~Akbari, M.~Einollahzadeh, M.~M. Karkhaneei, and M.~A. Nematollahi,
  \emph{Proof of a conjecture on the {S}eidel energy of graphs}, European J.
  Combin. \textbf{86} (2020), 103078, 8. \MR{4047519}

\bibitem{S_Akbari}
Saieed Akbari, Abdullah Alazemi, and Milica Andelic, \emph{Upper bounds on the
  energy of graph in terms of matching number}, Applicable Analysis and
  Discrete Mathematics \textbf{x} (2021), no.~x, 1--16.

\bibitem{Vertexenergy1}
Octavio Arizmendi, Jorge Fernandez~Hidalgo, and Oliver Juarez-Romero,
  \emph{Energy of a vertex}, Linear Algebra Appl. \textbf{557} (2018),
  464--495. \MR{3848283}

\bibitem{bap-kal-pat-weighted}
R.~B. Bapat, D.~Kalita, and S.~Pati, \emph{On weighted directed graphs}, Linear
  Algebra Appl. \textbf{436} (2012), no.~1, 99--111. \MR{2859913}

\bibitem{Tgain9}
Francesco Belardo, Maurizio Brunetti, Matteo Cavaleri, and Alfredo Donno,
  \emph{Godsil-{M}c{K}ay switching for mixed and gain graphs over the circle
  group}, Linear Algebra Appl. \textbf{614} (2021), 256--269. \MR{4209002}

\bibitem{Tgain8(line)}
Matteo Cavaleri, Daniele D'Angeli, and Alfredo Donno, \emph{Gain-line graphs
  via {$G$}-phases and group representations}, Linear Algebra Appl.
  \textbf{613} (2021), 241--270. \MR{4199816}

\bibitem{Cheng-Horn-Li}
Che-Man Cheng, Roger~A. Horn, and Chi-Kwong Li, \emph{Inequalities and
  equalities for the {C}artesian decomposition of complex matrices}, vol. 341,
  2002, Special issue dedicated to Professor T. Ando, pp.~219--237.
  \MR{1873621}

\bibitem{Day-So}
Jane Day and Wasin So, \emph{Graph energy change due to edge deletion}, Linear
  Algebra Appl. \textbf{428} (2008), no.~8-9, 2070--2078. \MR{2401641}

\bibitem{Godsil_Gutman}
C.~D. Godsil and I.~Gutman, \emph{On the theory of the matching polynomial}, J.
  Graph Theory \textbf{5} (1981), no.~2, 137--144. \MR{615001}

\bibitem{guo-mohar-jgt}
Krystal Guo and Bojan Mohar, \emph{Hermitian adjacency matrix of digraphs and
  mixed graphs}, J. Graph Theory \textbf{85} (2017), no.~1, 217--248.
  \MR{3634484}

\bibitem{Gutman}
Ivan Gutman, \emph{The energy of a graph}, Ber. Math.-Statist. Sekt. Forsch.
  Graz (1978), no.~100-105, Ber. No. 103, 22, 10. Steierm\"{a}rkisches
  Mathematisches Symposium (Stift Rein, Graz, 1978). \MR{525890}

\bibitem{Haemers}
Willem~H. Haemers, \emph{Seidel switching and graph energy}, MATCH Commun.
  Math. Comput. Chem. \textbf{68} (2012), no.~3, 653--659. \MR{3052170}

\bibitem{book-Jukna}
Stasys Jukna, \emph{Extremal combinatorics}, second ed., Springer-Verlag Berlin
  Heidelberg, 2011.

\bibitem{book_Gutman}
Xueliang Li, Yongtang Shi, and Ivan Gutman, \emph{Graph energy}, Springer, New
  York, 2012. \MR{2953171}

\bibitem{Lie}
Jianxi Liu and Xueliang Li, \emph{Hermitian-adjacency matrices and {H}ermitian
  energies of mixed graphs}, Linear Algebra Appl. \textbf{466} (2015),
  182--207. \MR{3278246}

\bibitem{Xiaobin-Ma}
Xiaobin Ma, \emph{A low bound on graph energy in terms of minimum degree},
  MATCH Commun. Math. Comput. Chem. \textbf{81} (2019), 393--404.

\bibitem{Miodrag-Bozin-Gutman}
Miodrag Mateljevi\'{c}, Vladimir Bo\v{z}in, and Ivan Gutman, \emph{Energy of a
  polynomial and the {C}oulson integral formula}, J. Math. Chem. \textbf{48}
  (2010), no.~4, 1062--1068. \MR{2726342}

\bibitem{Our-paper-1}
Ranjit Mehatari, M.~Rajesh Kannan, and Aniruddha Samanta, \emph{On the
  adjacency matrix of a complex unit gain graph}, Linear Multilinear Algebra
  \textbf{70} (2022), no.~9, 1798--1813. \MR{4429447}

\bibitem{mohar-secondkind}
Bojan Mohar, \emph{A new kind of {H}ermitian matrices for digraphs}, Linear
  Algebra Appl. \textbf{584} (2020), 343--352. \MR{4013179}

\bibitem{NIKIFOROV2009819}
Vladimir Nikiforov, \emph{More spectral bounds on the clique and independence
  numbers}, J. Combin. Theory Ser. B \textbf{99} (2009), no.~6, 819--826.
  \MR{2558437}

\bibitem{Mohammad_Reza_Oboudi}
Mohammad~Reza Oboudi, \emph{A very short proof for a lower bound for energy of
  graphs}, MATCH Commun. Math. Comput. Chem. \textbf{84} (2020), 345--347.

\bibitem{Pan-Chen-Li}
Yingui Pan, Jing Chen, and Jianping Li, \emph{Upper bound of graph energy in
  terms of matching number}, MATCH Commun. Math. Comput. Chem. \textbf{83}
  (2020), no.~2, 541--554.

\bibitem{Reff2016}
Nathan Reff, \emph{Oriented gain graphs, line graphs and eigenvalues}, Linear
  Algebra Appl. \textbf{506} (2016), 316--328. \MR{3530682}

\bibitem{our-paper3}
Aniruddha Samanta and M.~Rajesh Kannan, \emph{Bounds for the energy of a
  complex unit gain graph}, Linear Algebra Appl. \textbf{612} (2021), 1--29.
  \MR{4188336}

\bibitem{book-stanic}
Zoran Stanić, \emph{Inequalities for graph eigenvalues}, London Mathematical
  Society Lecture Note Series, Cambridge University Press, 2015.

\bibitem{F.Tian-D.Wong}
Fenglei Tian and Dein Wong, \emph{Upper bounds of the energy of triangle-free
  graphs in terms of matching number}, Linear Multilinear Algebra \textbf{67}
  (2019), no.~1, 20--28. \MR{3885877}

\bibitem{Wang-Ma}
Long Wang and Xiaobin Ma, \emph{Bounds of graph energy in terms of vertex cover
  number}, Linear Algebra Appl. \textbf{517} (2017), 207--216. \MR{3592020}

\bibitem{Wei-Li}
Wei Wei and Shuchao Li, \emph{Relation between the {H}ermitian energy of a
  mixed graph and the matching number of its underlying graph}, Linear
  Multilinear Algebra \textbf{68} (2020), no.~7, 1395--1410. \MR{4137053}

\bibitem{Wong-Wang-Chu}
Dein Wong, Xinlei Wang, and Rui Chu, \emph{Lower bounds of graph energy in
  terms of matching number}, Linear Algebra Appl. \textbf{549} (2018),
  276--286. \MR{3784349}

\bibitem{Tgain3(rank)}
Feng Xu, Qi~Zhou, Dein Wong, and Fenglei Tian, \emph{Complex unit gain graphs
  of rank 2}, Linear Algebra Appl. \textbf{597} (2020), 155--169. \MR{4080071}

\bibitem{Zaslav}
Thomas Zaslavsky, \emph{Biased graphs. {I}. {B}ias, balance, and gains}, J.
  Combin. Theory Ser. B \textbf{47} (1989), no.~1, 32--52. \MR{1007712}

\bibitem{signed-adj}
Thomas Zaslavsky, \emph{Matrices in the theory of signed simple graphs},
  Advances in discrete mathematics and applications: {M}ysore, 2008, Ramanujan
  Math. Soc. Lect. Notes Ser., vol.~13, Ramanujan Math. Soc., Mysore, 2010,
  pp.~207--229. \MR{2766941}

\end{thebibliography}

\end{document}